\documentclass[12pt]{amsart}

\usepackage{amssymb,amsthm, amsmath}
\usepackage{eucal,mathrsfs}
\usepackage{bbm}
\usepackage{color}
\usepackage[all]{xy}

\newtheoremstyle{mythm}%
  {\topskip}
  {\topskip}
  {\itshape}
  {}
  {\bfseries}
  {\\}
  {\parindent}
  {\thmname{#1}\thmnumber{ #2}\thmnote{ \textit{#3}}}
\newtheoremstyle{mydef}%
  {\topskip}
  {\topskip}
  {\itshape}
  {}
  {\bfseries}
  {\\}
  {\parindent}
  {\thmname{#1}\thmnumber{ #2}\thmnote{ \textit{#3}}}

\theoremstyle{mythm}
\newtheorem{lem}{Lemma}[subsection]
\newtheorem{thm}[lem]{Theorem}
\newtheorem*{thm*}{Theorem}

\newtheorem{cor}[lem]{Corollary}
\newtheorem*{cor*}{Corollary}
\newtheorem{prop}[lem]{Proposition}

\newtheorem{hyp}[lem]{Hypothesis}
\theoremstyle{mydef}
\newtheorem{defn}[lem]{Definition}
\newtheorem{warn}[lem]{Warning}
\newtheorem{rem}[lem]{Remark}

\newcommand{\mbbm}[1]{\mathbbm #1}
\newcommand{\mbb}[1]{\mathbb #1}
\newcommand{\mf}[1]{\mathfrak #1}
\newcommand{\mc}[1]{\mathcal #1}
\newcommand{\ms}[1]{\mathscr #1}

\newcommand{\oper}[1]{\operatorname{#1}}
\newcommand{\m}[1]{\mf{#1}}

\newcommand{\dra}{\dashrightarrow}

\newcommand{\Br}{\oper{Br}}
\newcommand{\W}{\oper{W}}

\newcommand{\Hom}{\oper{Hom}}

\newcommand{\cha}{\oper{char}}

\newcommand{\Spec}{\oper{Spec}}

\newcommand{\id}{\oper{id}}

\newcommand{\co}{\oper{R}}
\newcommand{\fracf}[1]{\oper{frac}(#1)}

\newcommand{\A}{\mbb A}
\newcommand{\F}{F}

\newcommand{\fk}{k}

\newcommand{\ov}{\overline}
\newcommand{\til}{\widetilde}
\newcommand{\wh}{\widehat}

\newcommand{\redcl}[1]{{#1}_{\fk}^{\text{red}}}
\newcommand{\points}[1]{S({#1})}
\newcommand{\components}[1]{\mc U({#1})}
\newcommand{\branches}[1]{\mc B({#1})}

\title{Field patching, factorization, and local-global principles}

\author{Daniel Krashen}

\begin{document}
\maketitle

\section{Introduction}

The goal of this paper is twofold - first to give an introduction to
the method of field patching as first presented in \cite{HH:FP}, and later used in \cite{HHK},
paying special attention to the relationship between
factorization and local-global principles and second, to extend the
basic factorization result in \cite{HHK} to the case of retract
rational groups, thereby answering a question posed to the author by J. L. Colliot-Th\'el\`ene.

Throughout, we fix a complete discrete valuation ring $T$ with field of
fractions $K$ and residue field $\fk$. Let $t \in T$ be a uniformizer.
Let $X/K$ be a smooth projective curve and $F$ its function field. 

Broadly speaking, the method of field patching is a procedure for
constructing new fields $F_\xi$ which will be in certain ways simpler
than $F$, and to reduce problems concerning $F$ to problems about the
various $F_\xi$. Overall, there are two ways in which this is done. Let
us suppose that we are interested in studying a particular type of
arithmetic object, such as a quadratic form, a central simple algebra,
etc.

\noindent
\textbf{Constructive Strategy (Patching) : } This consists in showing
that under suitable hypotheses, algebraic objects defined over the
fields $F_\xi$ which are ``compatible,'' exactly correspond to objects
defined over $F$ (see Theorem~\ref{object patching}). One may then use
this idea to construct new examples and counterexamples of such objects
by building them ``locally.''

\noindent
\textbf{Deconstructive Strategy (Local-global principle) : } We say
that a particular type of algebraic object satisfies a local-global
principle if whenever an object defined over $F$ becomes ``trivial''
when scalars are extended to each $F_\xi$, it must in fact have been
trivial to begin with (see Section~\ref{local global principles}). 

In this paper, we will not focus on these applications, which are
discussed for example in \cite{HH:FP, HHK, HHK:admissibility, CTPaSu}.
Instead, we focus on elucidating and extending the underlying methods
used.

\section*{}

The author is grateful to David Harbater, Julia Hartmann, and the
anonymous referee for numerous helpful comments and corrections in the
preparation of this manuscript.

\section{Patches and local-global principles}

\subsection{Fields associated to patches}

The fields $F_\xi$ are not canonically defined - they depend on a
number of choices, beginning with the choice of a model for $X$ over
$T$. 

\begin{defn} [Models] \label{regular model}
A \textbf{model} for the scheme $X/K$ is defined to be a connected normal
projective $\mbb P^1_T$-scheme $\wh X$ such that 
\begin{enumerate}
\item
the structure morphism $f : \wh X \to \mbb P^1_T$ is finite,
\item
considered as a $T$-scheme, the generic fiber $\wh X_K$ is isomorphic to
$X$,
\item The reduced closed fiber $\redcl{\wh X}$ is a normal crossings
divisor in $\wh X$,
\item $f^{-1} (\infty)$ contains all the singular points of the
reduced closed fiber $\redcl{\wh X}$.
\end{enumerate}
\end{defn}

Given a model $\wh X$ (we will generally suppress the morphism $f : \wh
X \to \mbb P^1_T$ from the notation), we let $\points{\wh X}$ denote the
set of closed points in $f^{-1} (\infty)$ and $\components{\wh X}$ denote
the set of connected (or equivalently, irreducible) components of
$\redcl{\wh X} \setminus \points{\wh X}$. These sets play a critical
role in what follows.

\begin{warn}
In other sources such as \cite{HH:FP, HHK, HHK:admissibility}, $\wh X$ is not given
the structure of a $\mbb P^1_T$-scheme, but rather the structure of a
$T$ scheme together with a distinguished set of closed points $S$. In
this context, one is allowed more general sets $S$. The reader must keep
in mind that a model $\wh X$ comes with the extra structure of a
morphism to $\mbb P^1$ throughout!
\end{warn}

It is perhaps a bit odd to include the finite morphism to $\mbb P^1_T$
as part of the definition of a model --- by comparison, in \cite{HH:FP},
it is only assumed that one should start with a projective $T$-curve
with a set $S$ of closed points such that there exists a finite
$T$-morphism to a curve with smooth reduced closed fiber and such that
the set $S$ is the inverse image of a set of closed points under this
morphism. We include the morphism to $\mbb P^1_T$ as part of our
definition simply as a matter of convenience of exposition. The following lemma shows that it is not much of an extra assumption, however: 
\begin{lem}[\cite{HH:FP}, Proposition~6.6]
Suppose $\wh X$ is a projective $T$-curve and $S \subset \wh X$ a finite set
of closed points. Then there exists a finite morphism $f : \wh X \to \mbb
P^1$ such that $S \subset f^{-1} (\infty)$.
\end{lem}

For the remainder of the section, we will suppose that we are given such
a model $\wh X$, and we let $F = \F(\wh X)$ be its function field. Given
any nonempty subset of points $Z \subset \wh X$, we define
\[ R_Z = \{ f \in F \mid \forall P \in Z, f \in \mc O_{\wh X, P} \} \]

We will define fields associated to two particular types of subsets $Z$:

\begin{defn} [Fields associated to closed points]
Let $P \in \wh X$ be a closed point. We define $R_P = R_{\{P\}} = \mc
O_{\wh X, P}$, $\wh R_P$ its completion with respect to its maximal ideal,
and $F_P$ the field of fractions of $\wh R_P$.
\end{defn}

\begin{defn} [Fields associated to open subsets of $\redcl{\wh X}$]
Let $U \subset \redcl{\wh X}$ be a nonempty irreducible Zariski open
affine subset of the reduced closed fiber  which is disjoint from the
singular locus of $\redcl{\wh X}$. We let $\wh R_U$ be the completion of
$R_U$ with respect to the ideal $t R_U$, and $F_U$ the field of fractions
of $\wh R_U$.
\end{defn}

Note that there are natural maps $F \subset F_U, F_P$ for any such $P$
and $U$, as well as inclusions $F_U \to F_V$ and $F_U \to F_P$,
whenever $V \subset U$ or $P \in U$ respectively.

\subsection{Some local-global principles} \label{local global
principles}

We may now give some examples of local-global principles. For these, we
assume that $X/K$ is a smooth projective curve where $K$ is a complete
discretely valued field with valuation ring $T$, and that we are given a
model $\wh X \to \mbb P^1$. We let $F$ be the function field of $X$.

\begin{thm}[Local-global principle for the Brauer group (see \cite{HH:FP}, Theorem~4.10)] \label{brauer local global}
Let $\Br(\cdot)$ denote the Brauer group. The natural homomorphism
\[ \Br(F) \to \left( \prod\limits_{P \in \points{\wh X}} \Br(F_P) \right) \times
\left( \prod\limits_{U \in \components{\wh X}} \Br(F_U) \right) \]
is injective.
\end{thm}
We give a proof of this result on page \pageref{brauerproof}.  In fact,
we will see later, using patching, that this may be extended to a three
term exact sequence by adding a term on the right (see
Theorem~\ref{brauer exact}).

\begin{thm}[Local-global principle for isotropy (see \cite{HHK}, Theorem~4.2)] \label{quadratic local global}
Suppose $q$ is a regular quadratic form of dimension at least 3, and
$\cha(F) \neq 2$. If $q_{F_P}$ and $q_{F_U}$ are isotropic for every $P
\in \points{\wh X}$ and $U \in \components{\wh X}$ then $q$ is also
isotropic.
\end{thm}

The proof of this is given on page \pageref{quadratic proof}. We remark
that same proof may be used to give the result in the case $\cha(F) = 2$
and $q$ even as well. See also Theorem~\ref{quadratic exact} for a
related result.

Both of these principles in fact, may be regarded as special cases of
either of the following results, the main new results of this paper:

\begin{defn}\label{transitive defn}
Suppose $H$ is a variety over $F$ and $G$ is an algebraic group which
acts on $H$. We say that $G$ \textbf{acts transitively} on $H$ if for
every field extension $L/F$, the group $G(L)$ acts transitively on the
set $H(L)$.
\end{defn}

The following result generalizes \cite{HHK}, Theorem~3.7 by weakening the
hypothesis of rationality to allow for retract rational groups as well:
\begin{thm}[Local-global principle for varieties with transitive actions]\label{homogeneous local global}
Suppose $G$ is a connected retract rational algebraic group defined
over $F$, and $H$ is a variety on which $G$ acts transitively. Then
$H(F) \neq \emptyset$ if and only if $H(F_P), H(F_U) \neq \emptyset$
for all $P \in \points{\wh X}$ and $U \in \components{\wh X}$.
\end{thm}

This theorem follows quickly from Theorem~\ref{main theorem}, and its
proof may be found just after the statement of this theorem on page
\pageref{main theorem}. The proof of this in the
case of retract rationality will occupy a good portion of this paper.
Along the way, we will explore the connections between this local-global
principles and the notion of ``factorization'' for the group $G$. The
following corollary is particularly useful. 

\begin{cor} \label{local global reductive homogeneous}
Suppose $G$ is a retract rational reductive group over $F$ and $H$ a
projective homogeneous variety for $G$. Then $H(F) \neq \emptyset$ if and
only if $H(F_P), H(F_U) \neq \emptyset$ for all $P \in \points{\wh X}$ and
$U \in \components{\wh X}$.
\end{cor}
\begin{proof}
This follows from the fact that the action of $G(F)$ on $H(F)$ is
transitive. This in turn in a consequence of \cite{Bo:LAG},
Theorem~20.9~(iii).
\end{proof}

From these theorems (or even the versions assuming only rationality of
$G$ from \cite{HHK}), we may prove the above local-global results
concerning the Brauer group and quadratic forms.

\begin{proof}[Proof of Theorem~\ref{brauer local global}]\label{brauerproof}
Let $\alpha \in \Br(F)$ and suppose $\alpha_{F_P} = 0$, $\alpha_{F_U} = 0$
for every $P \in \points{\wh X}$, $U \in \components{\wh X}$. We need to
show that $\alpha = 0$.

Let $A$ be a central simple $F$ algebra in the class of $\alpha$ and let
$H$ be the Severi-Brauer variety for $A$. Note that this is a
homogeneous variety for the group $GL(A)$ which is rational, connected
and reductive. Recall that for a field extension $L/F$, $H(L)$ is
nonempty exactly when $A \otimes_F L$ is a split algebra --- that is to
say, $\alpha_L = 0$. But since $\alpha_{F_P}, \alpha_{F_U} = 0$, we have
$H(F_P), H(F_U) \neq \emptyset$ for every $U, P$. Consequently, by
Corollary~\ref{local global reductive homogeneous}, it follows that
$H(F) \neq \emptyset$ and so $\alpha = 0$ as desired.
\end{proof}

\begin{proof}[Proof of Theorem~\ref{quadratic local
global}]\label{quadratic proof}
Let $q$ be a quadratic form over $F$ satisfying the hypotheses of the
Theorem. We wish to show that $q$ is isotropic. Let $H$ be the quadratic
hypersurface of projective space defined by the equation $q = 0$. Recall
that this is a homogeneous variety for the group $SO(q)$ which under the hypotheses is
a rational, connected,
reductive group (see \cite{BofInv}, page 209, excercise 9). As above, we immediately see that since
$H(F_P), H(F_U)$ are nonempty for each $P \in \points{\wh X}$ and $U \in
\components{\wh X}$, we have by Corollary~\ref{local global reductive
homogeneous}, $H(F) \neq \emptyset$ as desired.
\end{proof}

\section{Patching}

The fundamental idea of patching is that defining an algebraic object
over the field $F$ is equivalent to defining objects over each of the
fields $F_P$ for $P \in \points{\wh X}$ and $F_U$ for $U \in
\components{\wh X}$, together with the data of how these objects agree
on overlaps. This will be stated in this section in terms of an
equivalence of categories. We will simply cite the results of
\cite{HH:FP} section 6 and 7 for the most part, but we focus more on the 
equivalence of tensor categories, and explore how to produce other
examples of algebraic patching.

Suppose we are given a model $\wh X$ for a curve $X/K$.  Given a point
$P \in \points{\wh X}$, the height $1$ primes of $R_P$ which contain $t$
correspond to the components of $\redcl{\wh X}$ incident to $P$. Each
such component is the closure of a uniquely determined element $U \in
\components{\wh X}$. 

\begin{defn}[Branches, and their fields]
Given such a height $1$ prime $\mc P$ of $R_P$, corresponding to an
element $U \in \components{\wh X}$, a \textbf{branch} along $U$ at $P$ is an
irreducible component of the scheme $\wh R_P / \mc P \wh R_P$.
Alternately, these are in correspondence with the height one primes of
$\wh R_P$ containing $\mc P \wh R_P$. Given such a height $1$ prime
$\wp$, we let $\wh R_\wp$ be the $t$-adic completion of the localization
of $\wh R_P$ at $\wp$, and $F_\wp$ its field of fractions.  We let
$\branches{\wh X}$ denote the set of all branches at all points in
$\points{\wh X}$.
\end{defn}

The fields $F_P$ and $F_U$ come equipped with natural inclusions into
$F_\wp$ which we now describe.  We note that the natural inclusion $\wh
R_P \to \wh R_\wp$ induces an inclusion of fields $F_P \to F_\wp$.
Further, we note that $\wh R_\wp$ is a $1$ dimensional regular local ring,
and hence a DVR, whose valuation is determined by considering order of
vanishing along the branch corresponding to $\wp$. In particular,
considering the inclusion $F \subset F_P \subset F_\wp$, we find that all
the elements of $R_U$, can not have poles along any branch lying along
$U$, and in particular, we see we have an inclusion $R_U \subset \wh
R_\wp$. Since the $t$-adic topology on $\wh R_\wp$ is the same as the
$\wp$-adic topology, we further find that $\wh R_\wp$ is $t$-adically
complete, and we therefore have an induced inclusion $F_U \to F_\wp$.

\subsection{Patching finite dimensional vector spaces}

\begin{defn}[Patching problems] \label{patching problems}
A \textbf{patching problem} is a collection $V_\xi$ for $\xi \in
\points{\wh X} \cup \components{\wh X}$, where $V_\xi$ is a finite
dimensional $F_\xi$ vector space together with a collection of
isomorphisms $\phi_\wp : V_P \otimes_{F_P} F_\wp \to V_U \otimes_{F_U}
F_\wp$ of $F_\wp$ vector spaces for every branch $\wp$ at $P$ on $U$. We
denote this problem by $(V, \phi)$.

\noindent
We define a \textbf{morphism of patching problems} $f : (V, \phi) \to
(W, \psi)$ to be a collection of homomorphisms $f_\xi : V_\xi \to
W_\xi$ such that whenever $\wp$ is a branch at $P$ lying on $U$, the
following diagram commutes:
\[\xymatrix{
V_P \otimes_{F_P} F_\wp \ar[rr]^{f_P \otimes F_\wp} \ar[d]_{\phi_\wp} &
& W_P \otimes_{F_P} F_\wp \ar[d]^{\psi_\wp} \\
V_U \otimes_{F_U} F_\wp \ar[rr]_{f_U \otimes F_\wp} & & W_U
\otimes_{F_U} F_\wp
}\]
\end{defn}

We see then that patching problems naturally form a category, which we
denote by $\mc P \mc P(\wh X, S)$. In fact, this category has a
$\otimes$-structure as well defined by $(V, \phi) \otimes (W, \psi) =
(V \otimes W, \phi \otimes \psi)$ where $(V \otimes W)_\xi = V_\xi
\otimes_{F_\xi} W_\xi$ and
\[(\phi \otimes \psi)_\wp : (V \otimes W)_P \otimes_{F_P} F_\wp \to (V
\otimes W)_U \otimes_{F_U} F_\wp\] 
is given by $\phi_\wp \otimes_{F_\wp} \psi_\wp$ via the above
identification. One may also verify that this monoidal structure is
symmetric and closed (see \cite{MacL:Cat}, VII.7).

\begin{defn}
If $V$ is a vector space over $F$, we let $(\til V, \mbb I)$ denote the
patching problem defined by $\til V_P = V_{F_P}$ and $\til V_U =
V_{F_U}$ and where $\mbb I_\wp$ is induced by the natural
identifications \[(V \otimes_F F_P) \otimes_{F_P} F_\wp = V_\wp = (V
\otimes_F F_U) \otimes_{F_U} F_\wp \]
\end{defn}

\begin{thm}\label{finite patching}[\cite{HH:FP}, Theorem~6.4]
Consider the functor 
\[\Omega : \mc Vect_{\text{f.d.}}(F) \to \mc P \mc P(\wh X, S)\]
from the category of finite dimensional $F$-vector spaces to the
category of patching problems defined by sending a finite dimensional
vector space $V$ to the patching problem $(\til V, \mbb I)$. Then
$\Omega$ is an equivalence of categories.
\end{thm}

\subsection{Patching algebraic objects}

\begin{defn}
A \textbf{type of algebraic object} (generally abbreviated to simply a
``type'') is a symmetric closed monoidal category $\ms T$. If $\ms
T$ is a type and $L$ a field, then an algebraic object of type $\ms T$
over $L$ is a strict symmetric closed monoidal functor (see \cite{MacL:Cat}, \S VII.1, \S VII.7 and
\cite{Hov:MC} \S 4.1 for definitions) from the category $\ms T$ to the
category of finite dimensional vector spaces over $L$ (with its natural
symmetric closed monoidal structure).  Morphisms between algebraic
objects of type $\ms T$ are defined simply to be natural transformations
between functors. We let $\ms T(L)$ denote the category of such objects.
\end{defn}
Note that $\ms T$ in fact defines a (pseudo-)functor from the category
of fields to the $2$-category of categories (see \cite{Gray:2Cat} for definitions).

Despite the formality of this definition, one may observe that one may
interpret an algebraic object of a given type $\ms T$ to be given by a
vector space, or a collection of vector spaces, together with extra
structure encoded by perhaps a collection of morphisms between various
tensor powers of the vector spaces satisfying certain axioms, and where
morphisms between these objects are given by collections of linear maps
satisfying certain compatibilies with the extra structures given. For
example, we might consider:
\begin{itemize}
\item Lie algebras,
\item Alternative (or Jordan) algebras,
\item Operads,
\item Central simple algebras
\item Quadratic forms, where morphisms are isometries,
\item Quadratic forms, where morphisms are similarities,
\item Separable commutative or noncommutative algebras,
\item $G$-Galois extensions of rings in the sense of \cite{DeIn}
\item and so on...
\end{itemize}
In these cases, the category $\ms T$ in question is simply given as the
symmetric closed monoidal category generated by some set of objects
(corresponding to the underlying vector spaces of the structure) and
some morphisms (defining the structure of the algebra or form), such
that certain diagrams commute which define the structure in question.
For example, a central simple algebra is a vector space $A$ together
with a bilinear product $A \otimes A \to A$, and $F$-algebra structure
$F \to A$ such that the canonical ``sandwich map'' of algebras 
\[A \otimes A^{\text{op}} \to \Hom(A, A) \]
is an isomorphism (see \cite{DeIn}, Chapter~2, Theorem~3.4(iii)). In this case,
the category $\ms T$ is generated by a single element $a$, a morphism $a
\otimes a \to a$ and $\mbbm 1 \to a$ (where $\mbbm 1$ is the unit for the
monoidal structure), and such that the natural map $a \otimes a \to
\Hom(a, a)$ (where the $\Hom$ is defined by the closed structure) has an
inverse.

To see quadratic forms and isometries in this way, one may simply let the category $\ms T$ be generated by a
single element $v$ a morphism $v \otimes v \to \mbbm 1$, assumed to commute with the morphism switching the
order of the $v$'s. In the case of similarities instead of isometries, one may add a new object $\ell$, and
replace $v \otimes v \to \mbbm 1$ with a morphism $v \otimes v \to \ell$. To force $\ell$ to correspond to a
$1$-dimensional vector space, one may then add to this category an inverse to the natural morphism $\mbbm 1
\to \ell \otimes \ell^\ast \cong \Hom(\ell, \ell)$.

\begin{defn}[Patching problems] \label{object patching problems}
Let $\ms T$ be a type of algebraic object.  A \textbf{patching problem
of objects of type $\ms T$} is a collection $A_\xi$ for $\xi \in \points{\wh X} \cup
\components{\wh X}$, where $A_\xi$ is an object of type $\ms T$ over $F_\xi$,
together with a collection of isomorphisms $\phi_\wp : A_P
\otimes_{F_P} F_\wp \to A_U \otimes_{F_U} F_\wp$ in $\ms T(F_\wp)$.  We
denote this problem by $(A, \phi)$.
\end{defn}

Just as with vector spaces, we may define morphisms of patching
problems of objects of type $\ms T$, and again find that these form a
tensor category, which we denote $\mc P \mc P_{\ms T}(\wh X)$. Again
as before, if $A$ is an algebraic object of type $\ms T$ over $F$, we
may form a natural patching problem $(\til A, \mbb I)$, and obtain a
functor from $\ms T(F)$ to $\mc P \mc P_{\ms T}(\wh X)$.

\begin{thm}\label{object patching}
Consider the functor 
\[\Omega_{\ms T} : \ms T(F) \to \mc P \mc P_{\ms T}(\wh X)\] defined
by sending an algebraic object $A$ to the patching problem $(A, \mbb I)$. Then $\Omega_{\ms T}$ is an
equivalence of categories.
\end{thm}
\begin{proof}
Since we have an equivalence of categories $\mc Vect_{\text{f.d.}}(F)
\cong \mc P\mc P(\wh X)$ by Theorem~\ref{finite patching}, it is
immediate that this equivalence also induces an equivalence of functor
categories 
\[\ms T(F) = Fun(\ms T, \mc Vect_{\text{f.d.}}(F) \cong Fun(\ms T, \mc
P\mc P(\wh X)) \cong \mc P\mc P_{\ms T}(\wh X).\]
One may now check that this gives the desired equivalence.
\end{proof}

\begin{rem}
It would be interesting to know if one could extend this to equivalences of other kinds of
objects. In particular, infinite dimensional vector spaces, finitely generated commutative algebras, or
perhaps even to (some suitably restricted) categories of schemes. None of these fall under the definition of
an algebraic object given above, and it is therefore not at all clear if the conclusions of
Theorem~\ref{object patching} will still hold.
\end{rem}

\subsection{Central simple algebras and quadratic forms}

For the following results, we suppose we are given $\wh X$ a normal,
connected, projective, finite $\mbb P^1_T$-scheme. The machinery of
patching gives the exactness of various exact sequences relating to
field invariants derived from algebraic objects, such as the Brauer group
$\Br(F)$ and the Witt group $W(F)$ of quadratic forms over $F$.

\begin{thm}[(see \cite{HH:FP}, Theorem~7.2)] \label{brauer exact}
We have an exact sequence:
\[ 0 \to \Br(F) \to \left( \prod\limits_{P \in \points{\wh X}} \Br(F_P)
\right) \times \left( \prod\limits_{U \in \components{\wh X}} \Br(F_U)
\right) \to \prod\limits_{\wp \in \branches{\wh X}}
\Br(F_\wp).\]
\end{thm}
\begin{proof}
Exactness on the left was noted in Theorem~\ref{brauer local global}. To
see exactness in the middle, suppose we have classes $\alpha_P, \alpha_U$
such that $(\alpha_U)_{F_\wp} \cong (\alpha_P)_{F_\wp}$ whenever $\wp$
is a branch at $P$ on $U$. Since there are only a finite number of
points and components, we may choose an integer $n$ such that each of
the Brauer classes $\alpha_U, \alpha_P$ may be represented by central
simple algebras $A_U, A_P$ of degree $n$. Now, by hypothesis, for each
branch $\wp$ as above, we may find an isomorphism of central simple
algebras $\phi_\wp : (A_P)_{F_\wp} \to (A_U)_{F_\wp}$. But this gives the
data of a patching problem for central simple algebras, and therefore we
may find a central simple $F$-algebra $A$ such that $A_{F_P} \cong A_P$
and $A_{F_U} \cong A_U$ as desired.
\end{proof}

\begin{thm} \label{quadratic exact}
We have an exact sequence:
\[ \W(F) \to \left( \prod\limits_{P \in \points{\wh X}} \W(F_P)
\right) \times \left( \prod\limits_{U \in \components{\wh X}} \W(F_U)
\right) \ \to \prod\limits_{\wp \in \branches{\wh X}}
\W(F_\wp)\]
\end{thm}
\begin{proof}
The proof is very similar to the last one.  Suppose we have Witt classes
$\alpha_P, \alpha_U$ such that $(\alpha_U)_{F_\wp} =
(\alpha_P)_{F_\wp}$ whenever $\wp$ is a branch at $P$ on $U$. Since
there are only a finite number of points and components, we may choose
an integer $n$ such that each of the Witt classes $\alpha_U, \alpha_P$
may be represented by quadratic forms $q_U, q_P$ of the same dimension $n$.
Now, by hypothesis and Witt's cancellation theorems, for each branch $\wp$
as above, we may find an isometry $\phi_\wp : (q_P)_{F_\wp} \to
(q_U)_{F_wp}$. But this gives the data of a patching problem for quadratic
forms, and therefore we may obtain a form $q$ over $F$ such that the class
$\alpha$ of $q$ in $\W(F)$ has the property that $\alpha_{F_P} = \alpha_P$
and $\alpha_{F_U} = \alpha_U$.
\end{proof}

We note that exactness on the left is discussed in Theorem~\ref{quadratic local global}.

\subsection{Properties of $\wh R_P, \wh R_U, F_P, F_U$}

Let us now gather together some fundamental facts which we will need in the
sequel.

\begin{lem}[\cite{HH:FP}, Lemma~6.2] \label{field tensors}
Suppose $\wh Y \to \wh X$ is a finite morphisms of projective, normal,
finite $\mbb P^1_T$-schemes. Then the natural inclusions of fields yield
isomorphisms:
\[ 
F_P \otimes_{\F(\wh X)} F(\wh Y) \cong \prod F_{P'} \ \ \ \
F_U \otimes_{\F(\wh X)} F(\wh Y) \cong \prod F_{U'} \ \ \ \
F_\wp \otimes_{\F(\wh X)} F(\wh Y) \cong \prod F_{\wp'}
\]
where $P'$ (resp. $U'$, $\wp'$) range over all the points (resp.
components, branches) lying over $P$ (resp. $U$, $\wp$).
\end{lem}

\begin{lem}[\cite{HH:FP}, Lemma~6.3] \label{intersection property}
Let $\wh X$ be a projective, normal, finite $\mbb P^1_T$-scheme. Then
the natural inclusions of fields yield an exact sequence of $F = \F(\wh
X)$-vector spaces:
\[ 0 \to F \to \left(\prod\limits_{P \in \points{\wh X}} F_P\right)
\times \left( \prod\limits_{U \in \components{\wh X}} F_U\right) \to
\prod\limits_{\wp \in \branches{\wh X}} F_\wp \]
\end{lem}

\begin{lem} \label{complete sum}
Let $\m V, \m W \subset \m U$ be $t$-adically complete $T$-modules, and
suppose that $\m V/t \m V + \m W/t \m W = \m U/t \m U$. Then $\m V + \m W
= \m U$.
\end{lem}
\begin{proof}
Suppose $u \in \m U$. Let $v_0 = w_0 = 0$. We will inductively
construct a sequence of elements $v_i \in \m V$, $w_i \in \m W$ such
that $v_i - v_{i+1} \in t^i\m V, w_i - w_{i+1} \in t^i\m W, v_i + w_i -
u \in t^i \m U$. By completeness, these will converge to elements $v
\in \m V, w \in \m W$ such that $v + w = u$.

Suppose we have constructed $v_i, w_i$ satisfying the above hypotheses.
Since $u - v_i - w_i \in t^i \m U$, we may write $u - v_i - w_i = t^i
r$. By hypothesis, we may write $r = v' + w' + tr'$ for some $v' \in \m
V, w' \in \m W, r' \in \m U$. Setting $v_{i+1} = v_i + t^i v',
w_{i+1} = w_i + t^i w'$, completes the inductive step.
\end{proof}

\begin{lem} \label{sum for P1}
Considering $\mbb P^1_T$, we have $\wh R_{\mbb A^1} + \wh R_{\infty} =
\wh R_\wp$, where $\wp$ is the unique branch at $\infty$.
\end{lem}
\begin{proof}
Using Lemma~\ref{complete sum}, we need only check that $\ov R_{\mbb
A^1} + \ov R_{\infty} = \ov R_\wp$, where
\[
\ov R_{\mbb A^1} \cong \wh R_{\mbb A^1}/ t \wh R_{\mbb A^1}, \ \ \
\ov R_{\infty} \cong \wh R_{\infty}/ t \wh R_{\infty}, \ \ \
\ov R_{\wp} \cong \wh R_{\wp}/ t \wh R_{\wp}
\]
But, we may compute $\ov R_{\mbb A^1} = \fk[\mbb A^1_{\fk}]$, 
$\ov R_\infty = 
\wh{\mc O_{\mbb P^1_{\fk, \infty}}}$, 
$\ov R_\wp = \fracf{\wh{\mc O_{\mbb P^1_{\fk, \infty}}}}$.
Writing $x$ for the coordinate function on the affine part of the
$\fk$-line, we may explicitly identify
\[\ov R_{\mbb A^1} = k[x], \ \ \ \ov R_\infty = k[[x^{-1}]], \ \ \
\ov R_\wp = k((x^{-1})), \]
and the result follows.
\end{proof}

\section{Local-global principles, factorization and patching}

Let $\wh X \to \mbb P^1$ be a model for $X/K$, and let $G$ be an
algebraic group defined over $F$.

\subsection{Local-global principles for rational points}

\begin{defn} \label{factorization definition 1}
We say that \textbf{factorization holds for} $G$, with respect to $\wh X$, if for
every tuple $(g_\wp)_{\wp \in \branches{\wh X}}$, there exist
collections of elements $g_P$ for each $P \in \points{\wh X}$ and $g_U$
for each $U \in \components{\wh X}$ such that whenever $\wp$ is a branch
at $P$ on $U$ we have
\[g_\wp = g_P g_U\]
with respect to the natural embeddings $F_P, F_U \to F_\wp$.
\end{defn}

\begin{defn} \label{local-global defn rational points}
We say that \textbf{the local-global principle holds} for an $F$ scheme $V$, with
respect to a model $\wh X$ if $X(F) \neq \emptyset$ holds if and only if
$X(F_P), X(F_U) \neq \emptyset$ for every $P \in \points{\wh X}$ and $U
\in \components{\wh X}$.
\end{defn}

\begin{defn}
Let $G$ be an algebraic group over $F$ and $H$ a scheme over $F$.  We
say that $H$ is a \textbf{transitive} $G$-scheme if $G$ acts transitively on
$H$ (see Definition~\ref{transitive defn}).
\end{defn}

\begin{prop} \label{transitive rational points}
If factorization holds for a group $G$, then the local-global principle
holds for all transitive schemes over $G$.
\end{prop}
\begin{proof}
We essentially follow the proof of Theorem~{3.7} in \cite{HHK}. Suppose
have a group $G$ such that factorization holds for $G$, and a transitive
$G$-scheme $H$. Suppose we are given points $x_P \in H(F_P)$ and $x_U \in
H(F_U)$ for all $P$ and $U$. We will show that $H(F) \neq \emptyset$.

By transitivity of the action, whenever $\wp$ is a branch at $P$ on $U$,
we may find an element $g_\wp \in G(F_\wp)$ such that $g_\wp(x_P)_{F_wp} =
(x_U)_{F_\wp}$. By hypotheses, we may find elements $g_P \in G(F_P)$ and
$g_U \in G(F_U)$ for every $P$ and $U$ such that $g_\wp = g_P g_U$
whenever $\wp$ is a branch at $P$ on $U$. In particular, by replacing
$x_P$ by $g_P^{-1} x_P$ and $x_U$ by $g_U x_U$ we may assume that our
points satisfy $(x_P)_{F_\wp} = (x_U)_{F_\wp}$.

Now, consider these points as morphisms 
\[x_P : \Spec(F_P) \to H, \ \  x_U : \Spec(F_U) \to H, \ \ x_\wp :
\Spec(F_\wp) \to H\]
where $x_\wp$ is the composition of either $x_P$ or $x_U$ with the
respective maps $\Spec(F_\wp) \to \Spec(F_P), \Spec(F_U)$. We claim that
the scheme theoretic image of these maps consists of the same point in
$H$, for all $P$, $U$, and $\wp$. To see this, note that if $\wp$ is a
branch at $P$ on $U$, then the commutativity of the diagram
\[\xymatrix @R=1pc{
& \Spec(F_P) \ar[dr]^{x_P} \\
\Spec(F_\wp) \ar[rr]^{x_\wp} \ar[rd] \ar[ru] & & H \\
& \Spec(F_U) \ar[ur]_{x_U}
}\]
shows that the image of each of the morphisms $x_U$, $x_P$, $x_\wp$ are
the same. But since the closed fiber $\redcl{\wh X}$ is connected, it
follows that we may inductively show the image of all the morphisms
corresponding to points, components or branches must coincide.
Let $\kappa$ be the residue field of this image point $h \in H$. Then we
have field maps
\[\xymatrix @R=.5pc{
& F_P \ar[ld] \\
F_\wp & & \kappa  \ar[lu] \ar[ld] \\
& F_U \ar[lu]
}\]
Using Lemma~\ref{intersection property}, we find that we obtain a map
$\kappa \to F$ which one may check must be a homomorphism of fields.
Therefore, we obtain a morphism $\Spec(F) \to H$ mapping onto the point
$h$ and so $H(F) \neq \emptyset$ as desired.
\end{proof}

\subsection{Local-global principles for algebraic objects and torsors}

\begin{defn}
We say that \textbf{the local-global principle holds for an algebraic object}
$A$ (of some given type) if for any algebraic object $B$ (of the same
type), we have $A \cong B$ if and only if $A_{F_P} \cong B_{F_P}$ and
$A_{F_U} \cong B_{F_U}$ for all $P, U$.  We say that the local-global
principle holds for a particular type of algebraic object if it holds for
all algebraic objects of this type.
\end{defn}

\begin{prop}
The local-global principle holds for an algebraic object $A'$ if and
only for every patching problem $(A, \phi)$ of algebraic objects such
that $A_P \cong (A')_{F_P}, A_U \cong (A')_{F_U}$ for all $P, U$, the
isomorphism class of $(A, \phi)$ is independent of $\phi$.
\end{prop}
\begin{proof}
Suppose that the local-global principle holds for $A'$, and let $(A, \phi), (A,
\psi)$ be two patching problems, such that $A_P \cong (A')_{F_P}$ and
$A_U \cong (A')_{F_U}$ for each $P, U$. Since we may patch algebraic
objects, we may find algebraic objects $B_1, B_2$ over $F$ whose patching
problems are equivalent to $(A, \phi), (A, \psi)$ respectively. Since
$(B_1)_{F_U} \cong A_{F_U} \cong (B_2)_{F_U}$ and similarly for $F_P$, we find
that by the local-global principle, $B_1 \cong B_2$. Therefore their
associated patching problems are isomorphic, implying $(A, \phi) \cong (A,
\psi)$ as desired.

Conversely, suppose that $(A, \phi)$'s isomorphism class is independent
of $\phi$ for every patching problem. Suppose we are given $A', B'$
be algebraic objects over $F$ with associated patching problems $(A,
\phi)$ and $(B, \psi)$ respectively. Suppose further that $(A')_{F_U}
\cong (B')_{F_U}$ and similarly for $F_P$.  Since $A_U \cong (A')_{F_U}
\cong (B')_{F_U} \cong B_U$ and $A_P \cong (A')_{F_P} \cong (B')_{F_P}
\cong B_P$ for all $U, P$ by definition, we may change $\psi$ via these
isomorphisms to find $(B, \psi) \cong (A, \psi')$ for some $\psi'$. But
therefore by hypothesis, $(A, \phi) \cong (A, \psi') \cong (B, \psi)$.
Since patching gives an equivalence of categories, we further conclude
$A' \cong B'$, completing the proof.
\end{proof}

\begin{rem}
Let $\ms T$ be a type of algebraic object, and $A$ is a particular object
of type $\ms T$. Let $\ms T_A$ denote the subclass of objects which are
isomorphic to $A$ (more precisely, $\ms T_A$ is the sub-pseudofunctor of
$\ms T$ which associates to every field extension $L/F$ the category of
algebraic objects of type $\ms T$ over $L$ which are isomorphic to the
object $A_L$). Then $\ms T_A$ satisfies the hypotheses of patching ---
i.e. we have an equivalence of categories between the category $\mc P\mc
P_{\ms T_A}(\wh X)$ and $\ms T_A(F)$  --- if and only if the
local-global principle holds for $A$. Note that in general $\ms T_A$ is
not a ``type of algebraic object,'' described by some monoidal category
in the sense described above.
\end{rem}

\begin{defn}
Let $G$ be an algebraic group over $F$. We say that \textbf{the local-global
principle holds for} $G$ if for $\alpha \in H^1(F, G)$, with
$\alpha_{F_P}, \alpha_{F_U}$ trivial for each $P$, $U$, we have $\alpha$
trivial.
\end{defn}

Note that since elements of $H^1(F, G)$ correspond to torsors for $G$, we
see immediately that the local-global principle will hold for $G$ if and
only if the local-global principle holds for all $G$-torsors, in the sense
of Definition~\ref{local-global defn rational points}. Since $G$-torsors
are transitive $G$-schemes, from Proposition~\ref{transitive rational
points}, we immediately obtain:

\begin{prop} \label{local global torsors}
Suppose $G$ is a linear algebraic group defined over $F$, and suppose
that factorization holds for $G$ with respect to $\wh X$. Then the local
global principle holds for $G$.
\end{prop}

\begin{prop} \label{local global factorization}
Suppose $A$ is an algebraic object of some type $\ms T$, whose
automorphism group is the linear algebraic group $G$. Then the following
are equivalent:
\begin{enumerate}
\item the local-global principle holds for $A$,
\item the local-global principle holds for $G$.
\item factorization holds for $G$,
\end{enumerate}
\end{prop}
\begin{proof}
Since $G$ is the automorphism group of $A$, by descent (see
\cite{Serre:LF}, X.\S2, Proposition 4),
we may identify $H^1(L, G_L) = \oper{Forms}(A_L)$, the pointed set of
twisted forms of $A_L$. In particular, it is immediate from the
definition that the local-global principle for $A$ is equivalent to the
local-global principle for $G$. 

Suppose we have a local-global principle for $A$, and consider a
collection of elements $g_\wp \in G(F_\wp)$.  Consider the patching
problem $(B, \phi)$ where $B_P = A_{F_P}, B_U = A_{F_U}$, and $\phi_\wp
= g_\wp$. By the local-global principle, this is isomorphic to the
patching problem $(\til A, \mbb I)$. By definition, we may find an
isomorphism $h : (B, \mbb I) \to (B, \phi)$.  Let $g_P = h_P^{-1}$ and
$g_U = h_U$. By definition of a morphism of patching problems, we find
that $g_\wp = g_P g_U$, and that $g_P \in Aut(B_P) = G(F_P)$, $g_U \in
Aut(B_U) = G(F_U)$ as desired.

Conversely, suppose we have factorization for $G$. In this case it is
immediate from Proposition~\ref{local global torsors} that the local
global principle must hold for $G$, completing the proof.
\end{proof}

\begin{rem}
Theorem~\ref{local global factorization} raises the question of whether
it would be possible to show the equivalence of the local-global
principle for a group $G$ and factorization for this group without the
presence of an algebraic object with $G$ as its automorphism group. This would give a converse to
Proposition~\ref{local global torsors}. In turn since $G$-torsors are, in particular, transitive $G$-schemes,
one would then also obtain a converse to Proposition~\ref{transitive rational points}.
\end{rem}

\section{Factorization for retract rational groups}

\subsection{Overview and preliminaries}
The goal of this section will be to prove the following theorem:

\begin{thm} \label{main theorem}
Suppose $\wh X$ is a connected normal finite $\mbb P^1_T$-scheme, with
function field $F$ and let $G$ be a connected retract rational algebraic group over
$F$. Then factorization holds for $G$ with respect to $\wh X$.
\end{thm}

Using this theorem, we may easily proceed to the proof of the local
global principle for schemes with transitive action stated earlier in Theorem~\ref{homogeneous local
global}: If $G$ is a connected retract rational group over $F$, then by
the theorem, factorization holds for $G$ with respect to $\wh X$. But
then by Proposition~\ref{transitive rational points}, the local-global
principle must hold for transitive $G$ schemes, as desired.

The proof of this theorem will occupy the remainder of the section. Our
strategy will be to reduce this to a more abstract factorization
problem, arising from the case when $\wh X = \mbb P^1_T$. 
Overall, the proof stategy is roughly parallel to that followed in \cite{HHK}, where retractions of open
subsets of affine space take the place of open subsets of affine space.

\begin{defn} \label{factorization definition 2}
Suppose we have commutative rings $F \subset F_1, F_2 \subset F_0$, and an algebraic
group $G$ over $F$.  We will say that \textbf{factorization holds} with respect to
$G, F, F_1, F_2, F_0$ if for every $g_0 \in G(F_0)$ there exist $g_1 \in G(F_1)$ and
$g_2 \in G(F_2)$ such that $g_0 = g_1 g_2$.
\end{defn}
Note that here we are omitting from the notation the homomorphism $G(F_i)
\to G(F_0)$ for $i = 1, 2$. Suppose $\wh X$ is a connected, normal, finite
$\mbb P^1_T$-scheme. In this case, we set $F = \F(\wh X)$, and we let
\[F_1 = \prod\limits_{P \in \points{\wh X}} F_P, \ \ \ F_2 =
\prod\limits_{U \in \components{\wh X}} F_U, \ \ \ F_0 =
\prod\limits_{\wp \in \branches{\wh X}} F_\wp \]

\begin{rem}
It follows immediately from the definitions that factorization holds for
the group $G$ with respect to $\wh X$ in the sense of
Definition~\ref{factorization definition 1} if and only if factorization
holds for $G, F, F_1, F_2, F_0$ in the sense of
Definition~\ref{factorization definition 2} where $F, F_1, F_2, F_0$ are as
above.
\end{rem}

Back to the somewhat more abstract setting, suppose that $F$ is some
field, and let $L$ be a finite dimensional commutative $F$-algebra.
Recall that if $G$ is a linear algebraic group scheme, we may define its
Weil restriction, also referred to as its corestriction or transfer, as
the linear algebraic group with the functor of points defined by:
\[\co_{L/F}G(R) = G(R \otimes_F L)\]
where $R$ ranges through all $F$-algebras (\cite{Gro:FGA}, Exp.\ 195, p.\ 13 for the definition and Exp.\ 221,
p.\ 19 for proof of existence). We note
that the corestriction in fact
comes from a Weil restriction functor from the category of
quasi-projective $L$-schemes to the category of quasi-projective
$F$-schemes, and that this functor takes open inclusions to open
inclusions, and takes affine space to affine space (of a different
dimension). In particular, it follows that the corestriction of a
rational (or retract rational) variety is itself rational (resp. retract
rational).

We note the following Lemma, which is a consequence of the
definition of the corestriction in terms of the functor of points given
above.
\begin{lem}
Let $F$ be a field, and suppose we are given rings $F \subset F_1, F_2
\subset F_0$, and a finite dimensional commutative $F$-algebra $L$. Let
$G$ be a linear algebraic group over $L$. Then factorization holds for $G,
L, L \otimes_F F_1, L \otimes_F F_2, L \otimes F_0$ if and only if it
holds for $\co_{L/F}G, F, F_1, F_2, F_0$.
\end{lem}

\begin{lem} \label{corestriction reduction} Suppose that we are given a
morphism of connected projective normal finite $\mbb P^1_T$-schemes $f :
\wh Y \to \wh X$. Let $L$ be the function field of $\wh Y$ and $F$ the
function field of $\wh X$. Then factorization holds for $G, \wh Y$ if
and only if it holds for $\co_{L/F} G, \wh X$.
\end{lem}
\begin{proof}
This follows immediately from the universal property of the Weil
restriction, together with Lemma~\ref{field tensors}.
\end{proof}

\begin{lem}
Let $F$ be the function field of $\mbb P^1_T$. Suppose that for every
connected retract rational group $G$ over $F$, factorization holds for $G$
with respect to $\mbb P^1_T$ (as in Definition~\ref{factorization
definition 1}). Then for every normal finite $\mbb P^1_T$-scheme $\wh X$
with function field $L$, and every connected retract rational group $H$
over $L$, factorization holds for $H$ with respect to $\wh X$.
\end{lem}
\begin{proof}
This follows immediately from Lemma~\ref{corestriction reduction}.
\end{proof}

As a consequence of this, in order to prove Theorem~\ref{main theorem},
we may restrict to the setting where $F$ is the function field of $\mbb
P^1_T$, and where $F_1 = F_\infty$, $F_2 = F_{\mbb A^1_k}$, and where
$F_0 = F_\wp$ is the field associated to the unique branch $\wp$ along
$\mbb A^1_\fk$ at $\infty$. We let $\wh R_0 = \wh R_\wp, \m V = \wh
R_\infty, \m W = \wh R_{\mbb A^1_\fk}$. For convenience, in the sequel
we will often refer to the following hypothesis for factorization.

\begin{hyp}[see \cite{HHK}, Hypothesis~2.4]\label{factorization hyp}
We assume that the complete discrete valuation ring $\wh R_0$ contains a
subring $T$ which is also a complete discrete valuation ring having
uniformizer $t$, and that $F_1, F_2 $ are subfields of $F_0$ containing
$T$.  We further assume that $\m V \subset F_1 \cap \wh R_0$, $\m W
\subset F_2 \cap \wh R_0$ are $t$-adically complete $T$-submodules
satisfying $\m V + \m W  = \wh R_0$. 
\end{hyp}

\begin{lem}
With respect to the scheme $\mbb P^1_T$ consider
$F = \F({\mbb P^1_T})$, $F_0 = F_\wp$, $F_1 = F_\infty$, $F_2 =
F_{\mbb A^1_\fk}$, $\wh R_0 = \wh R_\wp$, $\m V = \wh R_\infty, \m W =
\wh R_{\mbb A^1_\fk}$.
Then these rings and modules satisfy the Hypothesis~\ref{factorization
hyp}.
\end{lem}
\begin{proof}
The completeness of $\m V, \m W$ is satisfied by definition. The fact
that $\m V + \m W = \wh R_0$ follows from Lemma~\ref{sum for P1}.
\end{proof}

\subsection{Retractions -- basic definitions and properties}

Before attacking the problem of factorization directly, it is necessary
to collect some facts concerning retractions and retract rational
varieties. Retract rational varieties were introduced by Saltman in \cite{AlgHomage}.

\begin{defn}
We say that a variety $Y$ is a \textbf{retraction} of a variety $U$ if
there exist morphisms $i : Y \to U$ and $p : U \to Y$ such that $pi =
id_Y$.  We say that it is a \textbf{closed retraction} if $i$ is a
closed embedding.
\end{defn}
\begin{rem}
In the case of a closed retraction, we will occasionally abuse notation
by simply regarding $i$ as an inclusion.
\end{rem}

\begin{defn}
We say that $Y$ is a \textbf{rational retraction} of $U$ if there are
rational maps $i : Y \dra U$ and $p : U \dra Y$ such that $pi = id_Y$ on
some open set on which $pi$ is defined.
\end{defn}

\begin{defn}
We say a variety $Y$ is \textbf{retract rational} if it is a
rational retraction of $\mbb A^n$ for some $n$.
\end{defn}

In \cite{Sa:RR}, the property of a variety being retract rational was reinterpreted in terms of lifting of
torsors. Our methodology goes in an (a priori) different direction, focusing on the local geometry of retract
rational varieties from the point of view of adic topologies.

\begin{lem}[Rational retractions shrink to retractions] \label{retraction morphism}
Suppose $Y$ is a rational retraction of $U$ via rational maps $i, p$.
Then we may find dense open subsets $Y_0 \subset Y$ and $U_0 \subset U$
such that $i, p$ make $Y_0$ a retraction of $U_0$.
\end{lem}
\begin{proof}
We may find open subsets $\til Y \subset Y$ and $\til U \subset U$ such
that $i, p$ restrict to morphisms on these sets, i.e. we have:
\[\xymatrix{
    Y \ar@{-->}[r] & U \ar@{-->}[r] & Y \\
    \til Y \ar[u] \ar[ru]^i & \til U \ar[u] \ar[ru]^p &
    \til Y \ar[u]
}\]
We choose $Y' = i^{-1}p^{-1}\til Y \subset \til Y$. We note that $pi$ is
defined on $Y'$ and so by definition, we may find $Y_0 \subset Y'$ such
that $pi|_{Y_0} = id_{Y_0}$. 
Let $U_0 = p^{-1}(Y_0)$. Then we have $pi(Y_0) \subset Y_0$ and so
$i(Y_0) \subset p^{-1}(Y_0) = U$. Since $p(U_0) \subset Y_0$ by
definition, we have constructed the desired morphisms.
\end{proof}

\begin{lem}[Retractions shrink to closed retractions] \label{retraction closed}
Suppose $Y$ is a retraction of $U$ via morphisms $i, p$.  Then we may
find dense open subvarieties $Y_0 \subset Y$ and $U_0 \subset U$ such
that $Y_0$ is a closed retraction of $U_0$ via the restrictions of $i,
p$.
\end{lem}
\begin{proof}
Since we may identify $Y$ with the image of $i$ it follows that $Y$ is
locally constructible in $U$ [EGA 4-1, p. 239 (Chevalley's thm)]. By
[EGA 3-1, p. 12], it follows that $Y$ is the intersection of a closed
and an open set in $U$. By setting $U_0$ to be this open set, and $Y_0 =
Y \cap U_0$ it follows that $Y_0$ is closed in $U_0$. Now it is easy to
see that the restrictions of $i,p$ exhibit $Y_0$ as a retraction of
$U_0$.
\end{proof}

\begin{cor}[Rational retractions shrink to closed retractions] \label{retraction closed morphism}
Suppose $Y$ is a rational retraction of $U$ via rational maps $i, p$.
Then we may find dense open subsets $Y_0 \subset Y$ and $U_0 \subset U$
such that $i, p$ make $Y_0$ a closed retraction of $U_0$.
\end{cor}
\begin{proof}
This follows immediately from Lemmas \ref{retraction morphism} and
\ref{retraction closed}.
\end{proof}

The following lemma gives us a first hint that retractions inherit some
of the geometry of the larger spaces.
\begin{lem}[Retractions of smooth schemes are smooth] \label{smooth
retraction}
Suppose $Y$ is a retraction of a smooth scheme $U$. Then $Y$ is smooth. 
\end{lem}
\begin{proof}
This follows from the formal criterion for smoothness (see for example
\cite{EGA4p4} \S 17 or \cite{Ill:FDH} \S 2). From this formulation, in the language of \cite{Ill:FDH}, we must
show that if $S_0 \to S$ is a thickening, and $f : S_0 \to Y$ is a
morphism, then we must be able to find a cover $\{V_i\}$ of $S$ and
morphisms $g_i : V_i \to Y$ extending $f|_{S_0 \cap V_i}$. To see this,
we first use the smoothness of $U$ to find $\til g_i : V_i \to U$
extending $i \circ f|_{S_0 \cap V_i}$. Now we set $g_i = p \circ \til
g_i$. We then have 
\begin{align*}
g_i|_{S_0 \cap V_i} &= p \circ \til g_i|_{S_0 \cap V_i} \\
&= p \circ i \circ f|_{S_0 \cap V_i} \\
&= f|_{S_0 \cap V_i}
\end{align*}
as desired.
\end{proof}

\begin{lem}[Standard position for retractions]\label{standard retraction}
Suppose $Y$ is a $d$-dimensional variety which is a closed retraction of
an open subscheme $U \subset \mbb A^n$. We also suppose that with respect
to the inclusion of $Y$ in $\mbb A^n$, that $0 \in Y$. Then we may shrink
$U$ and choose coordinates on $U$ so that $Y$ smooth and is the zero locus of
polynomials $f_1, \ldots, f_{n-d}$ with
\[f_i = x_i + P_i\]
where the $x_i$'s are the coordinate functions on $\mbb A^n$ and $P_i$ is
a polynomial in the $x_j$'s, each of whose terms are of degree at least
$2$.  

Further, we may alter $i$ and $p$ defining the retraction so that the
morphism $i p : U \to Y \to U$ is given by
\[(x_1, \ldots, x_n) \mapsto (M_1 + Q_1, \ldots, M_n + Q_n) \]
where
\[M_i = \left\{
    \begin{matrix}
    0 && \text{if } 1 \leq i \leq n-d \\
    x_i && \text{if } n-d < i \leq n
    \end{matrix}
    \right.
\]
and $Q_i$ is a rational function in the variables $x_i$, regular on $U$,
such that $\left. \frac\partial{\partial x_j} Q_i\right|_0 = 0$ for all $i,j$.
\end{lem}
\begin{proof}
For purposes of skimmability, we have placed this proof at the end of the section.
\end{proof}

\subsection{Adic convergence of Taylor series}

The basic strategy for factorization will be to produce closer and
closer approximations to a particular factorization. In order to carry
this out, it is necessary to discuss notions of convergence and
approximations in the adic setting, paralleling the discussion of \cite{HHK}, Section~2.

Suppose $F_0$ is a field complete with respect to a discrete valuation
$v$ with uniformizer $t$, and let $|a| = e^{-v(a)}$ be a corresponding
norm.  Let $A = F_0[x_1, \ldots, x_{N}]$, $\mf m$ the maximal ideal at
$0$, $A_{\mf m}$ the local ring at $0$ and $\wh A = F_0[[x_1, \ldots,
x_{N}]]$ the complete local ring at $0$. For $I = (i_1, \ldots,
i_{N}) \in \mbb N^{N}$, we let $|I| = \sum_j i_j$.  Define for $r \in
\mbb R$, $r > 0$
\[\wh A_r = \left.\left\{ \sum_{I} a_I x^I \right| \lim_{|I| \to \infty}
|a_I| r^{|I|} = 0 \right\} \]
and for $f = \sum a_I x^I \in \wh A_r$, we set
\[ |f|_r = \sup_{I} |a_I| r^{|I|}. \]

We give $\mbb A^n(F_0)$ a norm via the supremum of the coordinates
\[ |(a_1, \ldots, a_N)| = \max_i\{|a_i|\} \]
and we let $D(a, r)$ be the closed disk of radius $r$ about $a \in \mbb
A^n(F_0)$ with respect to the induced metric. We note that since the values of the metric are discrete, this
disk is in fact both open and closed in the $t$-adic topology.

We note the following elementary lemma:
\begin{lem} \label{elementary lemma}
Suppose $a \in D(0, r)$, and $f, g \in \wh A_r$. Then
\begin{enumerate}
\item $f + g, fg \in \wh A_r$,
\item $|f + g|_r \leq \max\{|f|_r, |g|_r\}$,
\item \label{complete filtration} for every real number $M > 0$, the group 
\[\{f \in \wh A_r \mid |f|_r < M\} \subset \wh A_r\] 
is complete with respect to the filtered collection of subgroups $\mf
m^i \cap \wh A_r$,
\item $|f|_r$ is finite, 
\item \label{submultiplicative} $|fg|_r \leq |f|_r |g|_r$,
\item \label{shrink the radius} if $r' < r$, then $|f|_{r'} \leq
\max\{|f(0)|, \frac{r'}{r}|f|_r \}$,
\item $f(a)$ is well defined (i.e. is a convergent series), and
\item \label{bounded functions} $|f(a)| \leq |f|_r$, and if $f(0) = 0$
then $|f(a)| \leq |f|_r |a| r^{-1}$.
\end{enumerate}
\end{lem}

\begin{lem} \label{regular to adic}
Suppose $f \in A_{\mf m}$. Then for all $\varepsilon \geq |f(0)|$ with $\varepsilon > 0$, there exists
$r > 0$ such that $f \in \wh A_r$ and $|f|_r < \varepsilon$. Further, for any $\delta > 0$ we may choose $r <
\delta$.
\end{lem}
\begin{proof}
Write $f = g/h$, $g, h \in A$ with $h \not\in \mf m$. Since $A/\mf m$
is a field, we may find $h' \in A$ with $hh' - 1 = -b \in \mf m$.
Therefore, in $\wh A$, we have $f = gh'(\sum b^i)$. Since $g, h', b$ are
polynomials, they are in $A_r$ for any $r$. Further, by
Lemma~\ref{elementary lemma}(\ref{shrink the radius}), we may reduce
$r$ so that $|gh'|_r \leq |f(0)|$, and since $b(0) = 0$,
we may also ensure $|b|_r < 1$. In doing this, note that we may also ensure that $r < \delta$. We note that by
Lemma~\ref{elementary lemma}(\ref{submultiplicative}), $|b^i|_r < 1$.  Now, by Lemma~\ref{elementary
lemma}(\ref{complete filtration}), it follows that $|\sum b^i|_r < 1$. Therefore $|f|_r = |gh'\sum b^i|_r
< \varepsilon$ as desired.
\end{proof}

\begin{lem} \label{adic vs Zariski}
The $t$-adic topology on $\mbb A^N(F_0)$ is finer than the Zariski
topology.
\end{lem}
\begin{proof}
It suffices to show that if $p \in \mbb A^N(F_0)$ and $f$ is a polynomial
not vanishing on $p$, we may find a disk about $p$ on which $f$ is
nonvanishing. Without loss of generality, we may apply a translation
and assume that $p = 0$. Let $g = f - f(0)$. By Lemma~\ref{regular to
adic}, since $g(0) = 0$, we may find an $r > 0$ such that $f \in \wh
A_r$ and such that $|g|_r < |f(0)|$ (using $\varepsilon = |f(0)|$). In particular, if $a \in \mbb
A^N(F_0)$ with $|a| < r$, we have $|g(a)| \leq |g|_r < |f(0)|$ by
Lemma~\ref{elementary lemma}(\ref{bounded functions}). Therefore, for
such an $a$, $f(a) = g(a) + f(0) \neq 0$. Therefore $f$ does not vanish
on a disk of radius $r$ about the origin as desired.
\end{proof}

\begin{prop}[Linear approximations and error term] \label{revised taylor with error}
Suppose $f \in \wh A_r$ for $r \leq 1$. Write
\[ f = c_0 + L + P \text{, where } P(\vec x) =
\sum_{|\nu| \geq 2} c_\nu x^\nu, \]
and $L$ is a linear form with coefficients in $F_0$ and all $c_\nu \in
F_0$. Suppose $|L + P|_r \leq 1$. Let $0 < \varepsilon
\leq |t| r^2$, and suppose $a, h \in \mbb A^N(F_0)$ with $|h|,|a| \leq
\varepsilon$. Then
\[ |f(a + h) - f(a) - L(h) | \leq |t||h|.\]
\end{prop}
\begin{proof}
This proof is a very slight modification of Lemma~{2.2} in \cite{HHK}.
Choose a real number $s$ so that we may write $|h| = \varepsilon |t|^s$.  We may rearrange the quantity of
interest as:
\[f(a + h) - f(a) - L(h) = \sum_{|\nu| \geq 2} c_\nu \left((a +
h)^\nu - a^\nu \right).\]
Since the absolute value is nonarchimedean, it suffices to show that
for every term $m = c_\nu x^\nu$ with $|\nu| \geq 2$ we have
\[|m(a + h) - m(a)| \leq \varepsilon|t|^{s+1}.\]

For a given $\nu$ with $|\nu| \ge 2$, consider the expression $(x +
x')^\nu - x^\nu$, regarded as a homogeneous element of degree $j =
|\nu|$ in the polynomial ring
$F_0[x_1, \ldots, x_{N}, x_1', \ldots, x_{N}']$.
Since the terms of degree $j$ in $x_1, \ldots, x_{N}$ cancel, the result
is a sum of terms of the form $\lambda \ell$ where $\lambda$ is an
integer and $\ell$ is a monomial in the variables $x, x'$ with total
degree $d$ in $x_1, \ldots, x_{N}$ and total degree $d'$ in $x_1',
\ldots, x_{N}'$, such that $d + d' = j$ and $d < j$.  Hence $d' \ge 1$.
Consequently, for each term of this form, 
\[|\lambda \ell(a, h)| \leq |\ell(a, h)| \leq \varepsilon^d(\varepsilon
|t|^s)^{d'} = \varepsilon ^ {j} |t|^{sd'}
\leq \varepsilon^{j}|t|^s.\]
Since $(a + h)^\nu - a^\nu$ is a sum of such terms, and the norm is
nonarchimedean, we conclude $|(a + h)^\nu - a^\nu| \leq
\varepsilon^{j}|t|^s$.  

Since $m = c_\nu x^\nu$, it follows that  
\[|m(a + h) - m(a)| \leq |c_\nu|\varepsilon^{j} |t|^s  \leq
r^{-j} \varepsilon^{j} |t|^s.\]
Now $\varepsilon \leq |t|r^2$, so $\varepsilon^{j-1} \leq |t|^{j-1}r^{2j
- 2}$.  Since $|t| < 1$, $r \le 1$, and $j \geq 2$, we have
\[\varepsilon^{j-1} \leq |t|^{j-1}r^{j + j - 2} \leq |t|r^j.\]
Rearranging this gives the inequality $(\varepsilon/r)^j \leq
\varepsilon |t|$ and so $(\varepsilon/r)^j |t|^s \leq \varepsilon
|t|^{s+1}$. Therefore 
\[|m(a + h) - m(a)| \leq r^{-j} \varepsilon^j |t|^s \leq \varepsilon
|t|^{s+1} = |t| |h|,\]
as desired.

\end{proof}

\begin{lem}[Local bijectivity / Inverse function theorem] \label{local bijectivity}
Suppose $f : U \to V$ is a morphism between Zariski open subschemes of
$\mbb A^d_{F_0}$ containing the origin and such that $f(0) = 0$.
Suppose further, that after writing the coordinates of $f$ as power series
in $\wh A$, we have $f = (f_1, \ldots, f_d)$ with $f_i = x_i + Q_i$ and
$Q_i$ consisting of terms of degree at least $2$. Then we may find
$t$-adic neighborhoods $U' \subset U(F_0)$ and $V' \subset V(F_0)$ of $0$
such that $f$ maps $U'$ bijectively onto $V'$. Further, we may assume that
$U'$ and $V'$ are disks about the origin of equal radii.
\end{lem}
\begin{proof}
By Lemma~\ref{regular to adic}, since $f(0) = 0$, we may find $0 < r \leq
1$ such that $f \in \wh A_r$ and $|f|_r \leq 1$.  Choose $\varepsilon
\leq |t|r^2$ as in the statement of Proposition~\ref{revised taylor
with error} and such that $D_0(\varepsilon) \subset V(F_0)$ and
$D_0(\varepsilon) \subset U(F_0)$. Let $V' = D_0(\varepsilon) \subset
V(F_0)$ and $U' = D_0(\varepsilon) \subset U(F_0)$. 
We claim that for $b \in U'$, we have $|f(b)| \leq \varepsilon$ and so
$f(b) \in V'$. To see this, we note that $Q_i \in \wh A_r$ and $|Q_i|_r \leq 1$, and hence we may apply
Proposition~\ref{revised taylor with error} (with $0$ linear and constant term) to see that 
$|Q_i(b)| \leq |t||b| < |b| = \max\{|b_i|\}$.  By the nonarchimedean
property, this gives 
\begin{multline*}
|f(b)| = \max\{|f_i(b)|\} = \max\{|b_i + Q_i(b)| \} \leq \max\{|b_i|,
|Q_i(b)|\} \\ = \max\{|b_j|\} = |b|.
\end{multline*}

We consider first surjectivity.  Note that both $U'$ and $V'$ are both
closed and open.  Since they are closed in a complete metric space,
they contain all limits of their Cauchy sequences.  Let $a \in V'$, and
let $b_0 = 0$. We will inductively construct elements $b_i \in U'$ such
that $|f(b_i) - a| \leq \varepsilon |t|^i$.  In particular, since $|a|
\leq \varepsilon$, we have $|f(b_0) - a| = |a| \leq \varepsilon$.
Assuming we have constructed $b_{i-1}$, we let $h = a - f(b_{i-1})$,
and note $|h| \leq \varepsilon |t|^{i-1}$ by hypothesis, and $|b_{i-1}|
\leq \varepsilon$ since $b_{i-1} \in U'$.  Therefore, by
Proposition~\ref{revised taylor with error}, we have 
\[ |f(b_{i-1} + h) - f(b_{i-1}) - h| \leq |t||h| \leq \varepsilon |t|^i.
\]
By setting $b_i = b_{i-1} + h$, we find that, since $f(b_{i-1}) + h
= a$, we have 
\begin{align*}
|f(b_i) - a|  &= |f(b_{i-1} + h) - a| \\
 &= |f(b_{i-1} + h) - f(b_{i-1}) - h| \leq |t||h| \leq \varepsilon
 |t|^i
\end{align*}
as desired. Since $b_i$ is a Cauchy sequence, using the completeness of $U'$, we may set $b = \lim b_i \in U'$
and we find by continuity that $f(b) = a$ as desired.

Next, we consider injectivity. Suppose $a, b \in U'$, let $h = b - a$
and suppose $h \neq 0$. We need to show that $f(a) \neq f(b)$.  Since
the valuation is nonarchimedean we have $a, h \leq \varepsilon$.  Let
$E = f(a + h) - f(a) - h$. Then we find $|E| \leq |h||t|$ by
Proposition~\ref{revised taylor with error}. But this means in
particular that $|E + h| = |h|$ by the nonarchimedean triangle
inequality. Therefore
\[ |f(b) - f(a)| = |f(a + h) - f(a)| = |E + h| = |h| \neq 0 \]
so $f(b) \neq f(a)$ as desired.
\end{proof}

\begin{lem} \label{bijective to the ball}
Suppose that $Y$ is a $d$-dimensional variety which is a closed retraction
of an open subscheme $U \subset \mbb A^n$ in the standard form of
Lemma~\ref{standard retraction} with respect to morphisms $i, p$. Then we
may find a $t$-adic neighborhood $V'$ of $0 \in Y$ (regarding $Y$ as a
subscheme of $U$ via $i$) such that the composition $V' \to \mbb A^n(F_0)
\to \mbb A^d(F_0)$ is bijective onto a $t$-adic disk, where the last map
is given by projection onto the last $d$ coordinates.
\end{lem}
\begin{proof}
As in Lemma~\ref{standard retraction}, we suppose that $Y$ is the zero
locus of polynomials $f_1, \ldots, f_{n-d}$ with
\[f_i = x_i + P_i\]
where $P_i$ is a polynomial in the $x_j$'s each of whose terms are of degree at least $2$. Using
Lemma~\ref{regular to adic}, we may choose $0 < r \leq 1$ such that $f_i
\in \wh A_r$ and $|f_i|_r \leq 1$ for each of the finitely many functions
$f_i$.  Choose $\varepsilon \leq |t| r^2$ as in Proposition~\ref{revised
taylor with error}. Let $g : U \to \mbb A^d$ the projection onto the
last $d$ coordinates.
Let $U' \in \mbb A^d(F_0)$ be the disk about the origin of radius $\varepsilon$. Let $V'$ be the intersection
of $g^{-1}U'$ with $Y(F_0)$. 

Suppose $a \in V'$ and $b \in U'$ with with $a \neq b$ and $g(a) = g(b)$.
We claim that $b \not\in V'$. In particular, this would imply that
$g|_{V'}$ is injective. To see $b \not\in V'$, first let $h = b - a$. If
$g(a) = g(b)$, then by definition of $g$, the last $d$ coordinates of $a$
and $b$ must match. Since $a \neq b$, we therefore know that $x_i(h) \neq
0$ for some $i = 1, \ldots, n-d$ where $x_i$ is the $i$'th coordinate
function on $\mbb A^d$. We may therefore choose $i$ such that $|x_i(h)|$
has the largest possible value, and in particular, we then would have
$|x_i(h)| = |h|$.  But, estimating $|f_i(b) - f_i(a)| = |f_i(a+h) -
f_i(a)|$ using Proposition~\ref{revised taylor with error}, we find 
\[|f_i(a+h) - f_i(a) - x_i(h)| \leq |t||h|.\]
We claim that $|f_i(a+h) - f_i(a)| \geq |h|$ and in particular that
$f_i(a) \neq f_i(b)$. To see this must hold, assume by contradiction that
$|f_i(a+h) - f_i(a)| < |h|$. In this case, we have
\[|f_i(a+h) - f_i(a) - x_i(h)| = |x_i(h)|\]
since $|x_i(h)| = |h|$. Therefore we have $|h| \leq |t||h|$ which is a
contradiction since $|t| < 1$.  Therefore, $f_i(b) \neq f_i(a)$. Since
$V'$ lies in the zero locus of the functions $f_i$, we have $f_i(a) = 0
\neq f_i(b)$, and so $b \not\in V'$ as claimed. Therefore $g|_{V'}$ is
injective.

By construction, $g|_{V'}$ has image entirely in the ball of radius
$\epsilon$ in $\mbb A^d(F_0)$ about the origin. We claim that it in fact
surjects onto this ball (possibly after shrinking $\epsilon$).
For this, let $a \in U'$, and consider its image $b = g(a) \in \mbb
A^d(F_0)$. Using the form for the retraction in Lemma~\ref{standard
retraction}, we may apply Lemma~\ref{local bijectivity} to the
composition (shrinking $\varepsilon$ if necessary)
\[\xymatrix{
\mbb A^d \cap U \ar[r] & U \ar[r]^p & Y \ar[r] \ar@/^1pc/[rr]^{g|_{Y}} & U
\ar[r]_g & \mbb A^d
}\]
By Lemma~\ref{local bijectivity}, we may find an inverse image $b'$ of $b$ in
$\mbb A^d$ of norm less than $\varepsilon$. Consequently, by
definition, the image of $b'$ in $Y$ must actually live in $V'$, and
this is an inverse image for $a$ as desired.
\end{proof}

\begin{cor}\label{local bijection corollary}
In the notation of the previous lemma, we may choose $t$-adic
neighborhoods of the origin $U' \in \mbb A^d$ and $V' \in Y$ such that the
composition $U' \to U \to Y$ takes $U'$ bijectively to $V'$ and the
composition $V' \to Y \to U \to \mbb A^n \to \mbb A^d$ takes $V'$
bijectively to $U'$.
\end{cor}
\begin{proof}
By Proposition~\ref{standard retraction} and Lemma~\ref{local
bijectivity}, we may find $U' \subset \mbb A^d$, $V' \subset Y$ so that
the composition $U' \to V' \to U'$ is bijective.
By Lemma~\ref{bijective to the ball}, we may find $V'' \subset V'$ such
that $V'' \to U'$ in bijective onto a $t$-adic disk $U'' \subset U'$.
But now again the composition $U'' \to U''$ is bijective, and since
$V'' \to U''$ is also bijective, we find $U'' \to V''$ is bijective as
well.
\end{proof}

\subsection{Factorization}

\begin{thm} \label{abstract small factorization}
Under Hypothesis~\ref{factorization hyp}, let $f : \mbb A^d_{F_0} \times
\mbb A^d_{F_0} \dra \mbb A^d_{F_0}$ be an $F_0$-rational map that is
defined on a Zariski open set $U \subseteq \mbb A^d_{F_0} \times \mbb
A^d_{F_0}$ containing the origin $(0,0)$. 
Suppose further that we may write: 
\begin{gather*}
f = (f_1, \ldots, f_d), \ \ f_i \in \wh \fk[x_1, y_1 \ldots, x_d,
y_d]_{\mf m} \\ 
\text{where } f_i = x_i + y_i + \sum_{|(\nu, \rho)| \geq 2} c_{\nu,
\rho, i} x^\nu y^\rho.
\end{gather*}
Then there is a real number $\varepsilon > 0$ such
that for all $a \in \mbb A^d(F_0)$ with $|a| \leq \varepsilon$, there
exist $v \in \m V^d$ and $w \in \m W^d$ such that $(v,w) \in U(F_0)$ and
$f(v, w) = a$.
\end{thm}

\begin{proof}
The proof of this theorem is exactly as in \cite{HHK}, Theorem~{2.5},
wherein in the first paragraph, the problem is reduced to exactly the
hypotheses which we assume.
\end{proof}

\begin{thm} \label{abstract retract rational factorization}
Under Hypothesis~\ref{factorization hyp}, let $m : Y \times Y \to Y$ be
a rational $F$-morphism defined at $(0, 0)$, and suppose that $m(y, 0) =
y = m(0, y)$ where it is defined. Suppose that $Y$ is a closed retraction of an
open subscheme of $\mbb A^n$. Then there exists $\varepsilon > 0$ such
that for $y \in Y(F_0) \subset \mbb A^n(F_0)$, $|y| \leq \varepsilon$, there
exist $y_i \in Y(F_i)$, $i = 1, 2$ such that $y = m(y_1, y_2)$. 
\end{thm}
\begin{proof}
We consider as in Corollary~\ref{local bijection corollary}, $t$-adic neighborhoods of $0$
$U' \subset \mbb A^d(F_0)$ and $V' \subset Y(F_0)$ such that we have
bijections $U' \to V'$ and $V' \to U'$ defined by algebraic rational
morphisms $p' : \mbb \A^d \dra Y$ and $i' : Y \dra \mbb A^d$. We consider
\[\xymatrix{
    V' \times V' \ar[rr]^{m|_{V' \times V'}} & & V' \ar[d]_{i'} \\
    U' \times U' \ar[u]^{p'} \ar[rr] & & U' \\
}\]
By hypothesis, the composition in the bottom $U' \times U' \to U'$ is
given as the restriction of an algebraic rational morphism $\mu : \mbb A^d
\times \mbb A^d \to \mbb A^d$. By Corollary~\ref{local bijection
corollary}, it is sufficient to show that $\mu$ is surjective when
restricted to a sufficiently small $t$-adic neighborhood.

We first shrink $V', U'$ if necessary to make them contained in Zariski
neighborhoods $V, U$ as in Lemma~\ref{standard retraction}.  Now, we
note that the rational map $\mu|_{\mbb A^d \times \{0\}}$ is just $ip$,
since $m|_{Y \times \{0\}} = id_Y$. By Lemma~\ref{standard retraction},
we find
\[m|_{\mbb A^d \times \{0\}} (x_1, \ldots, x_d) = (x_1 + Q_1, \ldots,
x_d + Q_d)\]
where $Q_i$ is a rational function in the variables $x_i$, regular on $U$,
such that $\left. \frac\partial{\partial x_j} Q_i\right|_0 = 0$ for all $i,j$. But now
we are done, using Theorem~\ref{abstract small factorization}.
\end{proof}

\begin{thm}[Factorization for retract rational groups] \label{retract rational factorization}
Under Hypothesis~\ref{factorization hyp}, assume that $F = \F(\mbb
P^1_T)$, $F_1 = F_{\mbb A^1_\fk}$, $F_2 = F_{\infty}$, and $F_0 =
F_\wp$, where $\wp$ is the unique branch at $\infty$.  Let $G$ be a
retract rational connected linear algebraic group defined over $F$. Then
for any $g_0 \in G(F_0)$ there exist $g_i \in G(F_i)$, $i = 1, 2$, such
that $g_1 g_2 = g_0$ --- that is to say, factorization holds for $G$ with
respect to $\mbb P^1_T$ (see Definition~\ref{factorization definition
1}).
\end{thm}
\begin{proof}
Using Lemma~\ref{retraction closed}, we may find an open subscheme $Y
\subset G$ which is a retraction of an open subscheme $U$ of affine
space. In particular, $Y$ must contain an $F$-rational point $y \in Y(F)$, and after replacing $Y$ by $y^{-1}
Y$ if necessary, we may assume $Y$ contains the identity element of $G$. Using \ref{abstract retract rational factorization}, where $m$ is
the multiplication map, we find that there exists $\varepsilon > 0$ such
that factorization holds for $g_0 \in G(F_0)$ provided that $|g_0| <
\varepsilon$. Fix such an epsilon, and suppose $g_0 \in G(F_0)$ is an
arbitrary element. Since $G$ is retract rational, it follows that $G(F)$
is Zariski dense in $G(F_0)$. Therefore, we have the existence of an
element $g' \in G(F)$ such that ${g'}^{-1}g_0 \in Y$. Since $Y$ is a
retraction of affine space, it follows that $Y(F_2)$ is
$t$-adically dense in $Y(F_0)$. Therefore we may find $g'' \in Y(F_2)$
such that $|{g'}^{-1} g_0 {g''}^{-1}| < \varepsilon$. Writing
$g'^{-1}g_0 g''^{-1}=
g_1 g_2$ where $g_i \in G(F_i)$, we conclude that $g_0 = (g'g_1)(g_2 g'')$.
Since $g'g_1 \in G(F_1)$ and $g_2g'' \in G(F_2)$, we are done.
\end{proof}

By Lemma~\ref{corestriction reduction} and the comments just following,
we conclude that Theorem~\ref{main theorem} holds.

\subsection{Proof of Lemma~\ref{standard retraction}}

\begin{lem} \label{trivial derivative means quadratic}
Suppose $f = g/h$ for $g, h \in k[x_1, \ldots, x_n]$ with $h(0) \neq 0,
g(0) = 0$ and $(\partial f/\partial x_i)  |_0 = 0$ for all $i$. Then if $R$ is a
$k$-algebra with $h(0) \in R^*$ and containing an element $\epsilon \in
R$, $\epsilon^2 = 0$ then $f(\epsilon v) = 0$ for $v \in k^n$.
\end{lem}
\begin{proof}
Since $g(0) = 0$, we may write $g = L + Q$ where $L$ is a linear
polynomial, and $Q$ is a sum of homogeneous terms of degree at least
$2$. Now we simply note that
\[\frac{\partial f}{\partial x_i} = \frac{h (\partial L/\partial x_i +
\partial Q/\partial x_i) - (L + Q) (\partial h/\partial x_i)}{h^2}\]
and in particular since $h(0) \neq 0$, we find $(\partial f/\partial
x_i) |_0 = 0$ implies that $h(0)(\partial L/\partial x_i(0)) = 0$,
which implies that all the coefficients of the linear form $L$ are $0$
and so $L = 0$. Since $h(0) \neq 0$, it follows that $h(\epsilon v)$ is
a unit, and we therefore may note that $f(\epsilon v) = Q(\epsilon
v)/h(\epsilon v)$ is well defined and $\epsilon^2 = 0$ implies
$Q(\epsilon v) = 0$, showing that $f(\epsilon v) = 0$ as desired.
\end{proof}

We now proceed with the proof of Lemma~\ref{standard retraction}.
By Lemma~\ref{smooth retraction}, we may assume that $Y$ is smooth.
Choose $f_1, \ldots, f_r$ which are regular on a neighborhood
of $0 \in U$ and which cut out $Y$. Writing $f_i = g_i/h_i$, for $g_i$ and
$h_i$ with no common factors, we see that since the $h_i$ don't vanish
at $0$, after shrinking $U$ so that the $h_i$ don't vanish on $U$, we may
ensure that the $h_i$ are units, and hence $Y$ is cut out by the $g_i$.
Therefore we may assume (after replacing $f_i$ by $g_i$ and shrinking
$U$), that the $f_i$ are polynomials. Next, we write 
\[f_i = L_i + P_i\]
where $L_i$ is a linear polynomial and $P_i$ has degree at least $2$.
Note that $f_i$ has no constant term since it must vanish at $0$. Since
$Y$ is smooth of dimension $d$, by the Jacobian criterion, the $L_i$'s
(which we may identify with the gradient of $f_i$ at $0$), span a $n-d$
dimensional space. After relabelling, we may assume that $L_1, \ldots,
L_{n-d}$ give a basis for this space. Let $\til Y$ be the zero locus of
$f_1, \ldots, f_{n-d}$. Since $Y \subset \til Y$ we have the codimension
of $\til Y$ at $0$, $codim_0(\til Y) \leq codim(Y) = n-d$. By
construction, the Jacobian matrix of the defining equations for $\til Y$
at $0$ has rank $n-d$, and so by \cite{Eis:CA}, page 402, $n-d \leq
codim(\til Y)$. But then
\[ n-d \leq codim_0(\til Y) \leq codim(Y) = n-d\]
so $codim_0(\til Y) = n-d$ and also by the Jacobian criterion, we
conclude that $\til Y$ is smooth at $0$. We may therefore, after
shrinking $U$ assume that $\til Y$ is smooth, irreducible, and of the
same dimension as $Y$. But $Y \subset \til Y$ therefore implies $Y =
\til Y$, and in particular we may assume $r = n-d$, and the $L_i$ are
independent.

After choosing a new basis for $\mbb A^n$, it is clear that we may
assume $L_i = x_i$ while preserving our hypotheses.

Finally, consider the morphism $\gamma = ip : U \to U$ (where $i$ and $p$ are as in the definition of the
retraction), and write
$\gamma(\vec{x}) = (\gamma_1(\vec x), \ldots, \gamma_n(\vec x))$, where
each $\gamma_i$ is a regular function on $U$. Since $\gamma_i(0) = 0$,
we may write $\gamma_i = M_i + Q_i$ for the linear function 
\[M_i = \sum \left.\frac\partial{\partial x_i} \gamma_i \right|_{\vec x
= 0} x_i \]
and have all the partial derivatives of the $Q_i$ vanishing.
Let $\mbb T = \Spec k[\epsilon]/(\epsilon^2)$, and consider a $\mbb T$-valued point $\tau :
\mbb T \to U$ given by $\vec a \epsilon = (a_1 \epsilon, \ldots, a_n
\epsilon) \in \mbb A^n(k[\epsilon]/(\epsilon^2))$. We note that $\tau$
maps $\mbb T$ into $Y$ if and only if $f_i(\vec a \epsilon) = 0$ for
each $i$. But we have (by Lemma~\ref{trivial derivative means quadratic})
\[f_i(\vec a \epsilon) = L_i(\vec a \epsilon) = \epsilon L_i(\vec a).\]
In particular, this occurs exactly when $a_i = 0$ for $1 \leq i \leq
n-d$. Since $\gamma(\vec a \epsilon) \in Y$, we therefore have $M_i = 0$
for $1 \leq i \leq n-d$. Since $\gamma|_Y = \id_Y$, looking on $\mbb
T$-valued points of $Y$ under $\gamma$, we find $M_i = M_i' + x_i$ for
$n-d < i \leq n$ where $M_i'$ is a linear function of $x_1, \ldots,
x_{n-d}$. Consider the linear function $\mbb A^n \to \mbb A^n$ given by
\[\phi : (x_1, \ldots x_n) \mapsto (y_1, \ldots, y_n) \]
where
\[ y_i = \left\{
    \begin{matrix}
    x_i & \text{if } 1 \leq i \leq n-d \\
    x_i - M_i & n-d < i \leq n
    \end{matrix}
    \right.
\]
Define rational maps $i' = \phi \circ i : Y \dra \mbb A^n$ and $p' = p
\circ \phi^{-1} : \mbb A^n \dra Y$. We then have $p' \circ i' = p \circ
\phi^{-1} \phi i = pi = id_Y$ as rational maps, and therefore define a
rational retraction. By Lemma~\ref{retraction closed} we may shrink $U$
and $Y$ to make this a closed retraction. Note also that $i'p' = \phi ip
\phi^{-1} =
\phi \gamma \phi^{-1}$. 

As before, let 
$\tau : \mbb T \to U$ be given by $\vec a \epsilon = (a_1 \epsilon,
\ldots, a_n \epsilon) \in \mbb A^n(k[\epsilon]/(\epsilon^2))$, where $\vec a \in \mbb A^n(k)$. We 
consider the morphism $i' p' : U \to Y \to U$, which we write as 
\[(x_1, \ldots, x_n) \mapsto (N_1 + P_1, \ldots, N_n + P_n) \]
with $N_i$ linear and the first derivatives of the $P_i$ vanishing at the
origin.  Computing using Lemma~\ref{trivial derivative means quadratic}
applied to functions $P_i$, we then find
\[i' p'(\tau) = \epsilon (N_1(\vec a), \ldots, N_n(\vec a)) \]
and also, using the linearity of $\phi$ and the fact that $i'p' = 
\phi \gamma \phi^{-1} = \phi \circ (M + Q) \circ \phi^{-1}$,
\begin{align*}
i' p'(\tau) &= \epsilon \phi (M_1(\phi(\vec a), \ldots, M_n(\phi^{-1}(\vec
a))) \\
&= \epsilon (0, \ldots, 0, a_{n-d+1}, \ldots, a_n),
\end{align*}
and so 
\[ N_i = \left\{
    \begin{matrix}
    0 & \text{if } 1 \leq i \leq n-d \\
    x_i & \text{if } n-d < i \leq n
    \end{matrix}
    \right.
\]
Therefore, upon replacing $p, i$ by $p', i'$, we obtain the desired
conclusion.

\def\cprime{$'$} \def\cprime{$'$} \def\cprime{$'$} \def\cprime{$'$}
  \def\cftil#1{\ifmmode\setbox7\hbox{$\accent"5E#1$}\else
  \setbox7\hbox{\accent"5E#1}\penalty 10000\relax\fi\raise 1\ht7
  \hbox{\lower1.15ex\hbox to 1\wd7{\hss\accent"7E\hss}}\penalty 10000
  \hskip-1\wd7\penalty 10000\box7}

\end{document}